\documentclass[reqno]{amsart}
\usepackage{amssymb,setspace,graphicx}
\usepackage{ifpdf}
\ifpdf
 \usepackage[hyperindex]{hyperref}
\else
 \expandafter\ifx\csname dvipdfm\endcsname\relax
 \usepackage[hypertex,hyperindex]{hyperref}
 \else
 \usepackage[dvipdfm,hyperindex]{hyperref}
 \fi
\fi
\allowdisplaybreaks[4]
\numberwithin{equation}{section}
\DeclareMathOperator{\td}{d\mspace{-2mu}}
\theoremstyle{plain}
\newtheorem{thm}{Theorem}[section]
\newtheorem{lem}{Lemma}[section]

\newtheorem{prop}{Proposition}[section]

\theoremstyle{remark}

\theoremstyle{definition}
\newtheorem{dfn}{Definition}[section]

\begin{document}

\title{The geometric mean is a Bernstein function}

\author[F. Qi]{Feng Qi}
\address[F. Qi]{Department of Mathematics, School of Science, Tianjin Polytechnic University, Tianjin City, 300387, China}
\email{\href{mailto: F. Qi <qifeng618@gmail.com>}{qifeng618@gmail.com}, \href{mailto: F. Qi <qifeng618@hotmail.com>}{qifeng618@hotmail.com}, \href{mailto: F. Qi <qifeng618@qq.com>}{qifeng618@qq.com}}
\urladdr{\url{http://qifeng618.wordpress.com}}

\author[X.-J. Zhang]{Xiao-Jing Zhang}
\address[X.-J. Zhang]{Department of Mathematics, School of Science, Tianjin Polytechnic University, Tianjin City, 300387, China}
\email{\href{mailto: X.-J. Zhang <xiao.jing.zhang@qq.com>}{xiao.jing.zhang@qq.com}}

\author[W.-H. Li]{Wen-Hui Li}
\address[W.-H. Li]{Department of Mathematics, School of Science, Tianjin Polytechnic University, Tianjin City, 300387, China}
\email{\href{mailto: W.-H. Li <wen.hui.li@foxmail.com>}{wen.hui.li@foxmail.com}}

\begin{abstract}
In the paper, the authors establish, by using Cauchy integral formula in the theory of complex functions, an integral representation for the geometric mean of $n$ positive numbers. From this integral representation, the geometric mean is proved to be a Bernstein function and a new proof of the well known AG inequality is provided.
\end{abstract}



\thanks{This paper was typeset using \AmS-\LaTeX}

\maketitle

\section{Introduction}

We recall some notions and definitions.

\begin{dfn}[\cite{mpf-1993, widder}]
A function $f$ is said to be completely monotonic on an interval $I$ if $f$ has derivatives of all orders on $I$ and
\begin{equation}
(-1)^{n}f^{(n)}(t)\ge0
\end{equation}
for $x \in I$ and $n \ge0$.
\end{dfn}

The class of completely monotonic functions on $(0,\infty)$ is characterized by the famous Hausdorff-Bernstein-Widder Theorem below.

\begin{prop}[{\cite[p.~161, Theorem~12b]{widder}}]\label{Bernstein-Widder-Theorem-12b}
A necessary and sufficient condition that $f(x)$ should be completely monotonic for $0<x<\infty$ is that
\begin{equation} \label{berstein-1}
f(x)=\int_0^\infty e^{-xt}\td\alpha(t),
\end{equation}
where $\alpha(t)$ is non-decreasing and the integral converges for $0<x<\infty$.
\end{prop}

\begin{dfn}[\cite{compmon2, minus-one}]
A function $f$ is said to be logarithmically completely monotonic on an interval $I$ if its logarithm $\ln f$ satisfies
\begin{equation}
(-1)^k[\ln f(t)]^{(k)}\ge0
\end{equation}
for $k\in\mathbb{N}$ on $I$.
\end{dfn}

It has been proved in~\cite{CBerg, absolute-mon-simp.tex, compmon2, minus-one} that a logarithmically completely monotonic function on an interval $I$ must be completely monotonic on $I$.

\begin{dfn}[\cite{Schilling-Song-Vondracek-2010, widder}]
A function $f:I\subseteq(-\infty,\infty)\to[0,\infty)$ is called a Bernstein function on $I$ if $f(t)$ has derivatives of all orders and $f'(t)$ is completely monotonic on $I$.
\end{dfn}

The class of Bernstein functions can be characterized by

\begin{prop}[{\cite[p.~15, Theorem~3.2]{Schilling-Song-Vondracek-2010}}]
A function $f:(0,\infty)\to\mathbb{R}$ is a Bernstein function if and only if it admits the representation
\begin{equation}
f(x)=a+bx+\int_0^\infty\bigl(1-e^{-xt}\bigr)\td\mu(t),
\end{equation}
where $a,b\ge0$ and $\mu$ is a measure on $(0,\infty)$ satisfying
\begin{equation*}
\int_0^\infty\min\{1,t\}\td\mu(t)<\infty.
\end{equation*}
\end{prop}

In~\cite[pp.~161\nobreakdash--162, Theorem~3]{Chen-Qi-Srivastava-09.tex} and~\cite[p.~45, Proposition~5.17]{Schilling-Song-Vondracek-2010}, it was discovered that the reciprocal of any Bernstein function is logarithmically completely monotonic.

\begin{dfn}[\cite{Atanassov}]
If $f^{(k)}(t)$ for some nonnegative integer $k$ is completely monotonic on an interval $I$, but $f^{(k-1)}(t)$ is not completely monotonic on $I$, then $f(t)$ is called a completely monotonic function of $k$-th order on an interval $I$.
\end{dfn}

It is obvious that a completely monotonic function of first order is a Bernstein function if and only if it is nonnegative on $I$.

\begin{dfn}[{\cite[p.~19, Definition~2.1]{Schilling-Song-Vondracek-2010}}]
If $f:(0,\infty)\to[0,\infty)$ can be written in the form
\begin{equation}\label{dfn-stieltjes}
f(x)=\frac{a}x+b+\int_0^\infty\frac1{s+x}{\td\mu(s)},
\end{equation}
then it is called a Stieltjes function, where $a,b\ge0$ are nonnegative constants and $\mu$ is a nonnegative measure on $(0,\infty)$ such that
\begin{equation*}
\int_0^\infty\frac1{1+s}\td\mu(s)<\infty.
\end{equation*}
\end{dfn}

The set of logarithmically completely monotonic functions on $(0,\infty)$ contains all Stieltjes functions, see~\cite{CBerg} or~\cite[Remark~4.8]{Open-TJM-2003-Banach.tex}. In other words, all the Stieltjes functions are logarithmically completely monotonic on $(0,\infty)$.
\par
In the newly-published paper~\cite{psi-proper-fraction-degree-two.tex}, a new notion ``completely monotonic degree'' of nonnegative functions was naturally introduced and initially studied.
\par
We also recall that the extended mean value $E(r,s;x,y)$ may be defined as
\begin{align}
E(r,s;x,y)&=\biggl[\frac{r(y^s-x^s)} {s(y^r-x^r)}\biggr]^{{1/(s-r)}}, & rs(r-s)(x-y)&\ne 0; \\
E(r,0;x,y)&=\biggl[\frac{y^r-x^r}{r(\ln y-\ln x)}\biggr]^{{1/r}}, & r(x-y)&\ne 0; \\
E(r,r;x,y)&=\frac1{e^{1/r}}\biggl(\frac{x^{x^r}}{y^{y^r}}\biggr)^{ {1/(x^r-y^r)}},& r(x-y)&\ne 0; \\
E(0,0;x,y)&=\sqrt{xy}\,, & x&\ne y; \\
E(r,s;x,x)&=x, & x&=y;\notag
\end{align}
where $x$ and $y$ are positive numbers and $r,s\in\mathbb{R}$. Because this mean was first defined in~\cite{stolarsky-mean}, so it is also called Stolarsky's mean. Many special mean values with two variables are special cases of $E$, for example,
\begin{align*}
E(r,2r;x,y)&=M_r(x,y), &&(\text{power mean or H\"older mean}) \\
E(1,p;x,y)&=L_p(x,y), &&(\text{generalized or extended logarithmic mean}) \\
E(1,1;x,y)&=I(x,y), &&(\text{identric or exponential mean}) \\
E(1,2;x,y)&=A(x,y), &&(\text{arithmtic mean}) \\
E(0,0;x,y)&=G(x,y), &&(\text{geometric mean}) \\
E(-2,-1;x,y)&=H(x,y), &&(\text{harmonic mean}) \\
E(0,1;x,y)&=L(x,y). &&(\text{logarithmic mean})
\end{align*}
For more information on $E$, please refer to the monograph~\cite{bullenmean}, the papers~\cite{emv-log-convex-simple.tex, Guo-Qi-Filomat-2011-May-12.tex, mon-element-exp-gen.tex}, and closely-related references therein.
\par
It is easy to see that the arithmetic mean
$$
A_{x,y}(t)=A(x+t,y+t)=A(x,y)+t
$$
is a trivial Bernstein function of $t\in(-\min\{x,y\},\infty)$ for $x,y>0$.
\par
It is not difficult to see that the harmonic mean
\begin{equation}
H_{x,y}(t)=H(x+t,y+t)=\frac2{\frac1{x+t}+\frac1{y+t}}
\end{equation}
for $t\in(-\min\{x,y\},\infty)$ and $x,y>0$ with $x\ne y$ meets
\begin{equation}\label{H-derivat-first}
H_{x,y}'(t)=\frac{2\bigl[x^2+y^2+2(x+y)t+2t^2\bigr]}{(x+y+2t)^2} =1+\frac{(x-y)^2}{(x+y+2t)^2}>1.
\end{equation}
It is obvious that the derivative $H_{x,y}'(t)$ is completely monotonic with respect to $t$. As a result, the harmonic mean $H_{x,y}(t)$ is a Bernstein function of $t$ on $(-\min\{x,y\},\infty)$ for $x,y>0$ with $x\ne y$.
\par
In~\cite[Remark~6]{new-upper-kershaw-JCAM.tex}, it was pointed out that the reciprocal of the identric mean
\begin{equation}
I_{x,y}(t)=I(x+t,y+t)=\frac1e\biggl[\frac{(x+t)^{x+t}}{(y+t)^{y+t}}\biggr]^{ {1/(x-y)}}
\end{equation}
for $x,y>0$ with $x\ne y$ is a logarithmically completely monotonic function of $t\in(-\min\{x,y\},\infty)$ and that the identric mean $I_{x,y}(t)$ for $t>-\min\{x,y\}$ with $x\ne y$ is also a completely monotonic function of first order (that is, a Bernstein function).
\par
In~\cite[p.~616]{gamma-psi-batir.tex}, it was concluded that the logarithmic mean
\begin{equation}
L_{x,y}(t)=L(x+t,y+t)
\end{equation}
is increasing and concave in $t>-\min\{x,y\}$ for $x,y>0$ with $x\ne y$. More strongly, it was proved in~\cite[Theorem~1]{log-mean-comp-mon.tex-mia} that the logarithmic mean $L_{x,y}(t)$ for $x,y>0$ with $x\ne y$ is a completely monotonic function of first order in $t\in(-\min\{x,y\},\infty)$, that is, the logarithmic mean $L_{x,y}(t)$ is a Bernstein function of $t\in(-\min\{x,y\},\infty)$.
\par
Recently, the geometric mean
\begin{equation}\label{G(x=y)(t)}
G_{x,y}(t)=G(x+t,y+t)=\sqrt{(x+t)(y+t)}\,
\end{equation}
was proved in~\cite{Geomeric-Mean-Complete-revised.tex} to be a Bernstein function of $t$ on $(-\min\{x,y\},\infty)$ for $x,y>0$ with $x\ne y$, and its integral representation
\begin{equation}\label{G(x-y)-integ-exp}
G_{x,y}(t)=G(x,y)+t+\frac{x-y}{2\pi} \int_0^\infty\frac{\rho((x-y)s)}s e^{-ys}\bigl(1-e^{-st}\bigr)\td s
\end{equation}
for $x>y>0$ and $t>-y$ was discovered, where
\begin{equation}\label{rho-funct-dfn}
\begin{aligned}
\rho(s)&=\int_0^{1/2} q(u)\bigl[1-e^{-(1-2u)s}\bigr]e^{-us}\td u\\
&=\int_0^{1/2} q\biggl(\frac12-u\biggr)\bigl(e^{us}-e^{-us}\bigr)e^{-s/2}\td u\\
&\ge0
\end{aligned}
\end{equation}
on $(0,\infty)$ and
\begin{equation}
q(u)=\sqrt{\frac1u-1}\,-\frac1{\sqrt{1/u-1}\,}
\end{equation}
on $(0,1)$.
\par
Let $a=(a_1,a_2,\dotsc,a_n)$ for $n\in\mathbb{N}$, the set of all positive integers, be a given sequence of positive numbers. Then the arithmetic and geometric means $A_n(a)$ and $G_n(a)$ of the numbers $a_1, a_2, \dotsc, a_n$ are defined respectively as
\begin{equation}
A_n(a)=\frac1n\sum_{k=1}^na_k
\end{equation}
and
\begin{equation}
G_n(a)=\Biggl(\prod_{k=1}^na_k\Biggr)^{1/n}.
\end{equation}
It is general knowledge that
\begin{equation}\label{AG-ineq}
G_n(a)\le A_n(a),
\end{equation}
with equality if and only if $a_1=a_2=\dotsm=a_n$.
\par
There has been a large number, presumably over one hundred, of proofs of the AG inequality~\eqref{AG-ineq} in the mathematical literature. The most complete information, so far, can be found in the monographs~\cite{bellman, bullenmean, hlp, kuang-3rd, mit, Mitrinovic-Vasic-1969} and a lot of references therein.
\par
In this paper, we establish, by using Cauchy integral formula in the theory of complex functions, an integral representation of the geometric mean
\begin{equation}
G_n(a+z)=\Biggl[\prod_{k=1}^n(a_k+z)\Biggr]^{1/n},
\end{equation}
where $a=(a_1,a_2,\dotsc,a_n)$ satisfies $a_k>0$ for $1\le k\le n$ and
$$
z\in\mathbb{C}\setminus(-\infty,-\min\{a_k,1\le k\le n\}].
$$
From this integral representation, it is immediately derived that the geometric mean $G_n(a+t)$ for $t\in(-\min\{a_k,1\le k\le n\},\infty)$ is a Bernstein function, where $a+t=(a_1+t,a_2+t,\dotsc,a_n+t)$, and a new proof of the AG inequality~\eqref{AG-ineq} is provided.

\section{Lemmas}

In order to prove our main results, we need the following lemmas.

\begin{lem}[Cauchy integral formula~{\cite[p.~113]{Gamelin-book-2001}}]\label{cauchy-formula}
Let $D$ be a bounded domain with piecewise smooth boundary. If $f(z)$ is analytic on $D$, and $f(z)$ extends smoothly to the boundary of $D$, then
\begin{equation}
f(z)=\frac1{2\pi i}\oint_{\partial D}\frac{f(w)}{w-z}\td w,\quad z\in D.
\end{equation}
\end{lem}

\begin{lem}\label{AG-New-lem=1}
For $z\in\mathbb{C}\setminus(-\infty,-\min\{a_k,1\le k\le n\}]$ with $a=(a_1,a_2,\dotsc,a_n)$ and $a_k>0$, the principal branch of the complex function
\begin{equation}\label{f(a-n)-dfn-eq}
f_{a,n}(z)=G_n(a+z)-z,
\end{equation}
where $a+z=(a_1+z,a_2+z,\dotsc,a_n+z)$, meets
\begin{equation}\label{AG-New-lem=1-lim}
\lim_{z\to\infty}f_{a,n}(z)=A_n(a).
\end{equation}
\end{lem}

\begin{proof}
By L'H\^ospital's rule in the theory of complex functions, we have
\begin{multline*}
\lim_{z\to\infty}f_{a,n}(z)=\lim_{z\to\infty}\biggl\{z\biggl[G_n\biggl(1+\frac{a}z\biggr)-1\biggr]\biggr\} \\
=\lim_{z\to0}\frac{G_n(1+az)-1}{z}
=\lim_{z\to0}\frac{\td}{\td z}\Biggl[\prod_{k=1}^n(1+a_kz)\Biggr]^{1/n}
=A_n(a),
\end{multline*}
where
$
1+\frac{a}z=\bigl(1+\frac{a_1}z,1+\frac{a_2}z,\dotsc,1+\frac{a_n}z\bigr)
$
and
$1+az=(1+a_1z,1+a_2z,\dotsc,1+a_nz)$. Lemma~\ref{AG-New-lem=1} is thus proved.
\end{proof}

\begin{lem}\label{Im-comput-AG-New-P-lem}
Let $a=(a_1,a_2,\dotsc,a_n)$ with $a_k>0$ for $1\le k\le n$ and let $[a]$ be the rearrangement of the positive sequence $a$ in an ascending order, that is, $[a]=\bigl(a_{[1]},a_{[2]},\dotsc,a_{[n]}\bigr)$ and $a_{[1]}\le a_{[2]}\le \dotsm \le a_{[n]}$.
For $z\in\mathbb{C}\setminus(-\infty,0]$, let
\begin{equation}\label{h(z)-dfn-eq}
h_n(z)=G_n\bigl([a]-a_{[1]}+z\bigr)-z,
\end{equation}
where $[a]-a_{[1]}+z=\bigl(z,a_{[2]}-a_{[1]}+z,\dotsc,a_{[n]}-a_{[1]}+z\bigr)$. Then the principal branch of $h_n(z)$ satisfies
\begin{multline}\label{Im-comput-AG-New-P-eq}
\lim_{\varepsilon\to0^+}\Im h_n(-t+i\varepsilon)=\\
\begin{cases}\displaystyle
\Biggl[\prod_{k=1}^n\bigl|a_{[k]}-a_{[1]}-t\bigr|\Biggr]^{1/n}\sin\frac{\ell\pi}n, & t\in\bigl(a_{[\ell]}-a_{[1]},a_{[\ell+1]}-a_{[1]}\bigr]\\
0, & t\ge a_{[n]}-a_{[1]}
\end{cases}
\end{multline}
for $1\le\ell\le n-1$.
\end{lem}

\begin{proof}
For $t\in(0,\infty)\setminus\bigl\{a_{[\ell+1]}-a_{[1]},1\le\ell\le n-1\bigr\}$ and $\varepsilon>0$, we have
\begin{multline*}
h_n(-t+i\varepsilon)=G_n\bigl([a]-a_{[1]}-t+i\varepsilon\bigr)+t-i\varepsilon\\
=\exp\Biggl[\frac1n\sum_{k=1}^n\ln\bigl(a_{[k]}-a_{[1]}-t+i\varepsilon\bigr)\Biggr]+t-i\varepsilon\\
=\exp\Biggl\{\frac1n\sum_{k=1}^n\bigl[\ln\bigl|a_k-a_{[1]}-t+i\varepsilon\bigr| +i\arg\bigl(a_{[k]}-a_{[1]}-t+i\varepsilon\bigr)\bigr]\Biggr\}+t-i\varepsilon\\
\to
\begin{cases}\displaystyle
\exp\Biggl(\frac1n\sum_{k=1}^n\ln\bigl|a_{[k]}-a_{[1]}-t\bigr| +\frac{\ell\pi}ni\Biggr)+t,  t\in\bigl(a_{[\ell]}-a_{[1]},a_{[\ell+1]}-a_{[1]}\bigr)\\ \displaystyle
\exp\Biggl(\frac1n\sum_{k=1}^n\ln\bigl|a_{[k]}-a_{[1]}-t\bigr| +\pi i\Biggr)+t,  t>a_{[n]}-a_{[1]}
\end{cases}\\
=
\begin{cases}\displaystyle
\Biggl(\prod_{k=1}^n\bigl|a_{[k]}-a_{[1]}-t\bigr|\Biggr)^{1/n} \exp\bigg(\frac{\ell\pi}ni\biggr)+t,  t\in\bigl(a_{[\ell]}-a_{[1]},a_{[\ell+1]}-a_{[1]}\bigr)\\ \displaystyle
\Biggl(\prod_{k=1}^n\bigl|a_{[k]}-a_{[1]}-t\bigr|\Biggr)^{1/n} \exp(\pi i)+t,  t>a_{[n]}-a_{[1]}
\end{cases}
\end{multline*}
as $\varepsilon\to0^+$. As a result, we have
\begin{multline*}
\lim_{\varepsilon\to0^+}\Im h_n(-t+i\varepsilon)=\\
\begin{cases}\displaystyle
\Biggl(\prod_{k=1}^n\bigl|a_{[k]}-a_{[1]}-t\bigr|\Biggr)^{1/n}\sin\frac{\ell\pi}n, & t\in\bigl(a_{[\ell]}-a_{[1]},a_{[\ell+1]}-a_{[1]}\bigr);\\
0, & t>a_{[n]}-a_{[1]}.
\end{cases}
\end{multline*}
\par
For $t=a_{[\ell+1]}-a_{[1]}$ for $1\le\ell\le n-1$, we have
\begin{multline*}
h_n(-t+i\varepsilon)=\exp\Biggl[\frac1n\sum_{k\ne\ell+1}^n \ln\bigl(a_{[k]}-a_{[1]}-t+i\varepsilon\bigr)+\frac1n\ln(i\varepsilon)\Biggr]+t-i\varepsilon\\
\begin{aligned}
&=\exp\Biggl[\frac1n\sum_{k\ne\ell+1}^n \ln\bigl(a_{[k]}-a_{[1]}-t+i\varepsilon\bigr)\Biggr] \exp\biggl[\frac1n\biggl(\ln|\varepsilon|+\frac\pi2i\biggr)\biggl]+t-i\varepsilon\\
&\to\exp\Biggl[\frac1n\sum_{k\ne\ell+1}^n \ln\bigl(a_{[k]}-a_{[1]}-t\bigr)\Biggr] \lim_{\varepsilon\to0^+}\exp\biggl[\frac1n\biggl(\ln|\varepsilon|+\frac\pi2i\biggr)\biggl]+t\\
&=t
\end{aligned}
\end{multline*}
as $\varepsilon\to0^+$. Hence, when $t=a_{[\ell+1]}-a_{[1]}$ for $1\le\ell\le n-1$, we have
\begin{equation*}
\lim_{\varepsilon\to0^+}\Im h_n(-t+i\varepsilon)=0.
\end{equation*}
The proof of Lemma~\ref{Im-comput-AG-New-P-lem} is completed.
\end{proof}

\section{The geometric mean is a Bernstein function}

We now turn our attention to establishing an integral representation of the geometric mean $G_n(a+z)$ and to showing that the geometric mean is a Bernstein function.

\begin{thm}\label{AG-New-thm1}
Let $a=(a_1,a_2,\dotsc,a_n)$ with $a_k>0$ for $1\le k\le n$ and let $[a]$ denote the rearrangement of the sequence $a$ in an ascending order, that is, $[a]=\bigl(a_{[1]},a_{[2]},\dotsc,a_{[n]}\bigr)$ and $a_{[1]}\le a_{[2]}\le \dotsm \le a_{[n]}$.
For $z\in\mathbb{C}\setminus(-\infty,-\min\{a_k,1\le k\le n\}]$, the principal branch of the geometric mean $G_n(a+z)$ has the integral representation
\begin{equation}\label{AG-New-eq1}
G_n(a+z)=A_n(a)+z-\frac1\pi\sum_{\ell=1}^{n-1}\sin\frac{\ell\pi}n \int_{a_{[\ell]}}^{a_{[\ell+1]}} \Biggl|\prod_{k=1}^n(a_k-t)\Biggr|^{1/n} \frac{\td t}{t+z},
\end{equation}
where $a+z=(a_1+z,a_2+z,\dotsc,a_n+z)$.
Consequently, the geometric mean $G_n(a+t)$ is a Bernstein function on $(-\min\{a_k,1\le k\le n\},\infty)$.
\end{thm}

\begin{proof}
By standard arguments, it is not difficult to see that
\begin{equation}\label{weighted-geometric-eq2-n}
\lim_{z\to0^+}[zh_n(z)]=0\quad \text{and}\quad h_n(\overline{z})=\overline{h_n(z)},
\end{equation}
where $h_n(z)$ is defined by~\eqref{h(z)-dfn-eq}.
\par
For any fixed point $z\in\mathbb{C}\setminus(-\infty,0]$, choose $0<\varepsilon<1$ and $r>0$ such that $0<\varepsilon<|z|<r$, and consider the positively oriented contour $C(\varepsilon,r)$ in $\mathbb{C}\setminus(-\infty,0]$ consisting of the half circle $z=\varepsilon e^{i\theta}$ for $\theta\in\bigl[-\frac\pi2,\frac\pi2\bigr]$ and the half lines $z=x\pm i\varepsilon$ for $x\le0$ until they cut the circle $|z|=r$, which close the contour at the points $-r(\varepsilon)\pm i\varepsilon$, where $0<r(\varepsilon)\to r$ as $\varepsilon\to0$. See Figure~\ref{note-on-li-chen-conj.eps}.
\begin{figure}[htbp]
\includegraphics[width=0.75\textwidth]{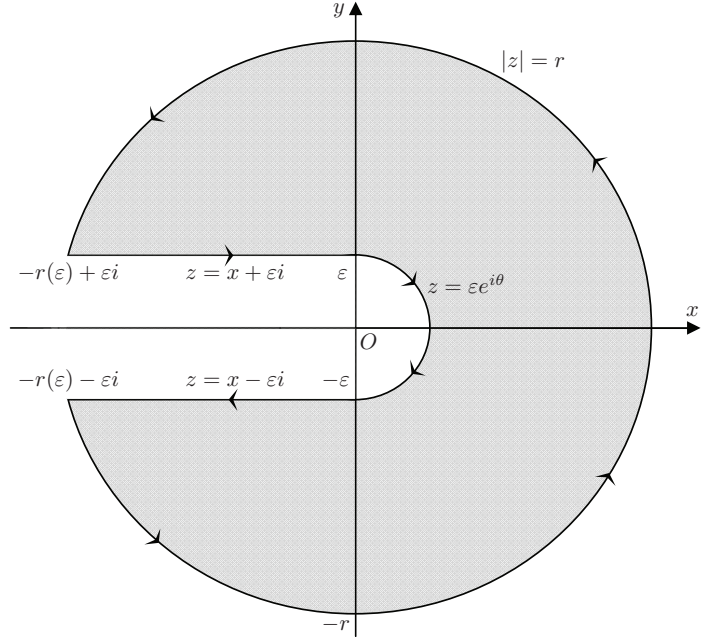}\\
\caption{The contour $C(\varepsilon,r)$}\label{note-on-li-chen-conj.eps}
\end{figure}
\par
By Cauchy integral formula, that is, Lemma~\ref{cauchy-formula}, we have
\begin{equation}
\begin{split}\label{h(z)-Cauchy-Apply}
h_n(z)&=\frac1{2\pi i}\oint_{C(\varepsilon,r)}\frac{h_n(w)}{w-z}\td w\\
&=\frac1{2\pi i}\biggl[\int_{\pi/2}^{-\pi/2}\frac{i\varepsilon e^{i\theta}h\bigl(\varepsilon e^{i\theta}\bigr)}{\varepsilon e^{i\theta}-z}\td\theta  +\int_{\arg[-r(\varepsilon)-i\varepsilon]}^{\arg[-r(\varepsilon)+i\varepsilon]}\frac{ir e^{i\theta}h\bigl(re^{i\theta}\bigr)}{re^{i\theta}-z}\td\theta\\
&\quad+\int_{-r(\varepsilon)}^0 \frac{h_n(x+i\varepsilon)}{x+i\varepsilon-z}\td x+\int_0^{-r(\varepsilon)}\frac{h_n(x-i\varepsilon)}{x-i\varepsilon-z}\td x\biggr].
\end{split}
\end{equation}
By the limit in~\eqref{weighted-geometric-eq2-n}, it follows that
\begin{equation}\label{zf(z)=0}
\lim_{\varepsilon\to0^+}\int_{\pi/2}^{-\pi/2}\frac{i\varepsilon e^{i\theta}h_n\bigl(\varepsilon e^{i\theta}\bigr)}{\varepsilon e^{i\theta}-z}\td\theta=0.
\end{equation}
By virtue of the limit~\eqref{AG-New-lem=1-lim} in Lemma~\ref{AG-New-lem=1}, we deduce that
\begin{equation}\label{big-circle-int=0}
\begin{split}
\lim_{\substack{\varepsilon\to0^+\\r\to\infty}} \int_{\arg[-r(\varepsilon)-i\varepsilon]}^{\arg[-r(\varepsilon)+i\varepsilon]}\frac{ir e^{i\theta}h_n\bigl(re^{i\theta}\bigr)}{re^{i\theta}-z}\td\theta
&=\lim_{r\to\infty}\int_{-\pi}^{\pi}\frac{ir e^{i\theta}h_n\bigl(re^{i\theta}\bigr)}{re^{i\theta}-z}\td\theta\\
&=2A_n\bigl([a]-a_{[1]}\bigr)\pi i,
\end{split}
\end{equation}
where $[a]-a_{[1]}=\bigl(0,a_{[2]}-a_{[1]},\dotsc,a_{[n]}-a_{[1]}\bigr)$.
Utilizing the second formula in~\eqref{weighted-geometric-eq2-n} and the limit~\eqref{Im-comput-AG-New-P-eq} in Lemma~\ref{Im-comput-AG-New-P-lem} results in
\begin{align}
&\quad\int_{-r(\varepsilon)}^0 \frac{h_n(x+i\varepsilon)}{x+i\varepsilon-z}\td x
+\int_0^{-r(\varepsilon)}\frac{h_n(x-i\varepsilon)}{x-i\varepsilon-z}\td x \notag\\
&=\int_{-r(\varepsilon)}^0 \biggl[\frac{h_n(x+i\varepsilon)}{x+i\varepsilon-z}-\frac{h_n(x-i\varepsilon)}{x-i\varepsilon-z}\biggr]\td x\notag\\
&=\int_{-r(\varepsilon)}^0\frac{(x-i\varepsilon-z)h_n(x+i\varepsilon) -(x+i\varepsilon-z)h_n(x-i\varepsilon)} {(x+i\varepsilon-z)(x-i\varepsilon-z)}\td x\notag\\
&=\int_{-r(\varepsilon)}^0\frac{(x-z)[h_n(x+i\varepsilon)-h_n(x-i\varepsilon)] -i\varepsilon[h_n(x-i\varepsilon)+h_n(x+i\varepsilon)]} {(x+i\varepsilon-z)(x-i\varepsilon-z)}\td x\notag\\
&=2i\int_{-r(\varepsilon)}^0\frac{(x-z)\Im h_n(x+i\varepsilon) -\varepsilon\Re h_n(x+i\varepsilon)} {(x+i\varepsilon-z)(x-i\varepsilon-z)}\td x\notag\\
&\to2i\int_{-r}^0\frac{\lim_{\varepsilon\to0^+}\Im h_n(x+i\varepsilon)}{x-z}\td x\notag\\
&=-2i\int^r_0\frac{\lim_{\varepsilon\to0^+}\Im h_n(-t+i\varepsilon)}{t+z}\td t\notag\\
&\to-2i\int^\infty_0\frac{\lim_{\varepsilon\to0^+}\Im h_n(-t+i\varepsilon)}{t+z}\td t\notag\\
&=-2i\sum_{\ell=1}^{n-1}\sin\frac{\ell\pi}n \int_{a_{[\ell]}-a_{[1]}}^{a_{[\ell+1]}-a_{[1]}} \Biggl[\prod_{k=1}^n\bigl|a_{[k]}-a_{[1]}-t\bigr|\Biggr]^{1/n} \frac{\td t}{t+z} \label{level0lines}
\end{align}
as $\varepsilon\to0^+$ and $r\to\infty$. Substituting equations~\eqref{zf(z)=0}, \eqref{big-circle-int=0}, and~\eqref{level0lines} into~\eqref{h(z)-Cauchy-Apply} and simplifying generate
\begin{equation}\label{h-n(z)=int}
h_n(z)=A_n\bigl([a]-a_{[1]}\bigr)-\frac1\pi\sum_{\ell=1}^{n-1}\sin\frac{\ell\pi}n \int_{a_{[\ell]}-a_{[1]}}^{a_{[\ell+1]}-a_{[1]}} \Biggl[\prod_{k=1}^n\bigl|a_{[k]}-a_{[1]}-t\bigr|\Biggr]^{1/n} \frac{\td t}{t+z}.
\end{equation}
From~\eqref{f(a-n)-dfn-eq} and~\eqref{h(z)-dfn-eq}, it is easy to obtain that
$$
f_{a,n}(z)=h_n\bigl(z+a_{[1]}\bigr)+a_{[1]}.
$$
Combining this with~\eqref{h-n(z)=int} and changing the variables of integrals, it is immediate to deduce that
\begin{align*}
f_{a,n}(z)&=A_n\bigl([a]-a_{[1]}\bigr)+a_{[1]}\\
&\quad-\frac1\pi\sum_{\ell=1}^{n-1}\sin\frac{\ell\pi}n \int_{a_{[\ell]}-a_{[1]}}^{a_{[\ell+1]}-a_{[1]}} \Biggl[\prod_{k=1}^n\bigl|a_{[k]}-a_{[1]}-t\bigr|\Biggr]^{1/n} \frac{\td t}{t+z+a_{[1]}}\\
&=A_n([a])-\frac1\pi\sum_{\ell=1}^{n-1}\sin\frac{\ell\pi}n \int_{a_{[\ell]}}^{a_{[\ell+1]}} \Biggl[\prod_{k=1}^n|a_{[k]}-t|\Biggr]^{1/n} \frac{\td t}{t+z},
\end{align*}
from which and the facts that
$$
A_n([a])=A_n(a)\quad \text{and}\quad \prod_{k=1}^n|a_{[k]}-t|=\prod_{k=1}^n|a_k-t|,
$$
the integral representation~\eqref{AG-New-eq1} follows.
\par
Differentiating with respect to $z$ on both sides of~\eqref{AG-New-eq1} yields
\begin{equation*}
G_n'(a+z)=1+\frac1\pi\sum_{\ell=1}^{n-1}\sin\frac{\ell\pi}n \int_{a_{[\ell]}}^{a_{[\ell+1]}} \Biggl[\prod_{k=1}^n|a_k-t|\Biggr]^{1/n} \frac{\td t}{(t+z)^2},
\end{equation*}
which implies that $G_n'(a+t)$ is completely monotonic, and so the geometric mean $G_n(a+t)$ is a Bernstein function. Theorem~\ref{AG-New-thm1} is proved.
\end{proof}

\section{A new proof of the AG inequality}

As an application of the integral representation~\eqref{AG-New-eq1} in Theorem~\ref{AG-New-thm1}, we can easily deduce the AG inequality~\eqref{AG-ineq} as follows.
\par
Taking $z=0$ in the integral representation~\eqref{AG-New-eq1} yields
\begin{equation}\label{AG-ineq-int}
G_n(a)=A_n(a)-\frac1\pi\sum_{\ell=1}^{n-1}\sin\frac{\ell\pi}n \int_{a_{[\ell]}}^{a_{[\ell+1]}} \Biggl[\prod_{k=1}^n|a_k-t|\Biggr]^{1/n} \frac{\td t}{t}\le A_n(a),
\end{equation}
from which the inequality~\eqref{AG-ineq} follows.
\par
From~\eqref{AG-ineq-int}, it is also immediate that the equality in~\eqref{AG-ineq} is valid if and only if $a_{[1]}=a_{[2]}=\dotsm=a_{[n]}$, that is, $a_1=a_2=\dotsm=a_n$. The proof of the AG inequality~\eqref{AG-ineq} is complete.

\end{document}